\documentclass[12pt,twoside,reqno,a4paper]{amsart}
\usepackage[latin1]{inputenc}
\usepackage[english]{babel}
\usepackage{mathrsfs}
\usepackage{indentfirst}
\usepackage{paralist}
\usepackage{amssymb}
\usepackage{amsthm}
\usepackage{amsmath}
\usepackage{amscd}
\usepackage{graphicx}
\usepackage[colorlinks=true]{hyperref}
\usepackage[margin=1.5in]{geometry}
\usepackage{cite}
\usepackage{lineno} 
\usepackage{enumitem} 
\usepackage{multicol}
\usepackage{xcolor}
\usepackage{upgreek}
\definecolor{ForestGreen}{rgb}{0.13, 0.55, 0.13}

\theoremstyle{plain}
\newtheorem{theorem}{Theorem}[section]
\newtheorem*{theorem*}{Main Result}
\newtheorem{thm*}{Known result}
\newtheorem{corollary}[theorem]{Corollary}
\newtheorem{lemma}[theorem]{Lemma}
\newtheorem{proposition}[theorem]{Proposition}
\newtheorem{definition}[theorem]{Definition}

\theoremstyle{definition}

\theoremstyle{remark}
\newtheorem{remark}[theorem]{Remark}

\numberwithin{equation}{section}

\newcommand{\eqlab}[1]{\begin{equation}  \begin{aligned}#1 \end{aligned}\end{equation}} 
\newcommand{\bgs}[1]{\begin{equation*} \begin{aligned}#1\end{aligned}\end{equation*}} 
\newcommand{\syslab}[2] []  {\begin{equation}#1  \left\{\begin{aligned}#2\end{aligned}\right.\end{equation}} 
\newcommand{\sys}[2][]{\begin{equation*}#1  \left\{\begin{aligned}#2\end{aligned}\right.\end{equation*}}

\newcommand{\R}{\ensuremath{\mathbb{R}}}
\newcommand{\Rn}{\ensuremath{\mathbb{R}^N}}
\newcommand{\N}{\ensuremath{\mathbb{N}}}
\newcommand{\eps}{\ensuremath{\varepsilon}}

\DeclareMathOperator{\loc}{loc}

\title[Quasilinear logarithmic Choquard equations]{Quasilinear logarithmic Choquard equations\\ with exponential growth in $\R^N$}

\author[C.~Bucur]{Claudia Bucur}
\author[D.~Cassani]{Daniele Cassani}
\address[C.~Bucur \and D.~Cassani]{\newline\indent Dip. di Scienza e Alta Tecnologia
	\newline\indent
	Universit\`{a} degli Studi dell'Insubria
	\newline\indent and
	\newline\indent RISM--Riemann International School of Mathematics
	\newline\indent Villa Toeplitz, Via G.B. Vico, 46 -- 21100 Varese}
    \email{\href{mailto:daniele.cassani@uninsubria.it}{daniele.cassani@uninsubria.it}}
     \email{\href{mailto:claudiadalia.bucur@uninsubria.it}{claudiadalia.bucur@uninsubria.it}
    }
\author[C.~Tarsi]{Cristina Tarsi}
\address[C.~Tarsi]{\newline\indent Dipartimento di Matematica
	\newline\indent
	Universit\`a degli Studi di Milano
	\newline\indent Via C. Saldini, 50 -- 20133 Milano
	}
\email{\href{mailto:cristina.tarsi@unimi.it}{cristina.tarsi@unimi.it}}

\thanks{Corresponding author: daniele.cassani@uninsubria.it}

\date{\today}

\begin{document}

\begin{abstract}
We consider the $N$-Laplacian Schr\"odinger equation strongly coupled with higher order fractional Poisson's equations. When the order of the Riesz potential $\alpha$ is equal to the Euclidean dimension $N$, and thus it is a logarithm, the system turns out to be equivalent to a nonlocal Choquard type equation. On the one hand, the natural function space setting in which the Schr\"odinger energy is well defined is the Sobolev limiting space $W^{1,N}(\R^N)$, where the maximal nonlinear growth is of exponential type. On the other hand, in order to have the nonlocal energy well defined and prove the existence of finite energy solutions, we introduce a suitable $log$-weighted variant of the Pohozaev-Trudinger inequality which provides a proper functional framework where we use variational methods.
\end{abstract}
\maketitle

\setcounter{tocdepth}{3}

 \section{Introduction and main results}
 \noindent Consider the following system of elliptic equations
 \begin{equation}\label{gen_sys}
\begin{cases}
-\Delta_m u+V(x)|u|^{m-2} u=f(u)v,\quad &\\
& x\in\R^N,\quad N\geq 2\\
-\Delta^{\frac{\alpha}{2}} v=F(u)\ ,
\end{cases}
\end{equation}
where $\Delta_m $, $m\geq 2$, is the $m$-Laplace operator defined as follows
	\[\Delta_m u = \mbox{div}(|\nabla u|^{m-2} \nabla u),\]
$V:\R^N\to\R$ is the external Schr\"odinger potential, $F$ is the primitive of $f$ vanishing at zero and where $(-\Delta)^{\frac{\alpha}{2}}$, $\alpha>0$, is the fractional Laplacian, see Section \ref{heu}. System \eqref{gen_sys} is in gradient form as the nonlinearity in the right hand side of \eqref{gen_sys} is the gradient of the potential function $G(u,v)=F(u)v$. It is also strongly coupled as $u=0\iff v=0$.  However, \eqref{gen_sys} does not possess in general a variational structure because of the presence of the nonlocal operator in the second equation, which prevents solutions of the system to be critical points of an energy functional $E(u,v)$, which may not exist or may not be well defined.

\noindent The parameters $m,N,\alpha$ play an important role from the theoretical point of view as well as from that of applications, see \cite{MV} and references therein. For $m=2$ we have the linear Schr\"odinger operator in the left hand side of the first equation and the problem has been widely studied in dimension $N>2$ and for $\alpha<N$, see \cite{MV} for a survey and \cite{CVZ,CZ} for related critical cases. The case of dimension $N=2$ and $\alpha<N$ has been studied in \cite{BV,ACTY}. More recently in \cite{BCV,CiWe,castar} it has been considered the limiting case $\alpha=N=2$; see also \cite{ACFM,car_med_rib} for related results.

\noindent A major difficulty in the limiting case is to construct a proper function space framework in which to settle the problem. As developed in \cite{castar} in dimension $N=2$, in order to consider the maximal exponential growth, a suitable functional framework can be obtained by means of $log$-weighted versions of the Pohozaev--Trudinger inequality \cite{SP,NT}.

\noindent In this paper we tackle the general limiting case
\begin{equation}\label{gen_case}
\boxed{\alpha=N=m>2}\ .
\end{equation}
\noindent This leads from one side to handle a quasilinear Schr\"odinger equation in the system \cite{CWZ} and on the other side demands for a more general function space setting. A key ingredient for this purpose, is to extend the fundamental functional inequality established in \cite{castar}, in the special case $\alpha=N=m=2$, to the general case \eqref{gen_case}.

\noindent Let $I_N\colon \Rn \setminus \{0\}\to \R$ be the logarithmic Riesz kernel
\[ I_N(x)= \frac{1}{\gamma_N}\log\frac{1}{|x|}\]
with
\[	\gamma_N=2^{N-1} \pi^{\frac{N}2} \Gamma\left(\frac{N}2\right).\]
By setting $v:=I_N*F(u)$, \eqref{gen_sys} is formally equivalent (see Section \ref{heu}) to the following quasilinear Choquard type equation
\eqlab{ \label{main} -\Delta_N u  + V |u|^{N-2} u = (I_N*F(u)) f(u) \qquad \mbox{ in } \Rn,}
\noindent which does have a variational structure.

\noindent Indeed, \eqref{main} is the Euler-Lagrange equation related to the energy functional
$$E(u)= \frac 1N \int_{\Rn} |\nabla u|^N + V |u|^N \, dx-\frac 12 \int_{\Rn} (I_N *F(u))(x) F(u(x))\, dx \ ,$$
provided such energy is well defined in a suitable function space which we are going to construct in the sequel as one of our main results.

\medskip

\noindent Before stating our main results let us introduce a few assumptions:
\begin{itemize}
\item[$(V)$] $V \colon \Rn \to \R$ is continuous, $1$-periodic and there exists $V_0>0$ such that $V(x)\geq V_0$;

\medskip

\item[$(f_1)$] $f\colon \R \to \R$ is continuous and differentiable, such that
\begin{itemize}
  \item[(i)] $f(s) \geq 0$, for all $s\geq 0$ and we may also assume (as we look for positive solutions) $f(s)=0$ for $s\leq 0$;
  \item[(ii)] there exists $C>0$ such that $f(s) \leq C s^p e^{\alpha_N s^{\frac{N}{N-1}}}$ as $s \to +\infty$, for some $p>0$ and where $\alpha_N$ is given below;
  \item[(iii)] $f(s) \asymp s^{q-1}$, as $s \to +\infty$, for some $q> N$;
\end{itemize}

\medskip

\item[$(f_2)$] there exist $C>\delta>0$ such that $$\frac{N-2}{N}+\delta \leq \frac{F(s) f'(s)}{f^2(s)}\leq C,\quad s> 0;$$

\medskip

\item[$(f_3)$] $\lim_{s\to +\infty}\frac{F(s) f'(s)}{f^2(s)}= 1$, or equivalently
	$\lim_{s\to +\infty}\frac{d}{d s}\frac{F(s)}{f(s)}=0$; 	
	
	\medskip
	
\item[$(f_4)$] there exists $\beta>0$ such that
$$\lim_{s\to +\infty}\frac{ s^{\frac{2N-1}{N-1}} f(s) F(s)}{e^{ 2N (\omega_{N-1} )^\frac{1}{N-1} s^{\frac{N}{N-1}} }} \geq \beta > \upnu,$$
    where $\upnu$ will be explicitly given in Sect. \ref{var_frame}.
\end{itemize}

\noindent Notice that from the assumptions on $f$ we also deduce the following:
\begin{itemize}
\item there exists $s_0 >1$ such that
\syslab[0 \leq F(s) \leq C ]{ \label{f11}
		& |s|^q, &&  s \leq s_0,\\
		& s^{p-\frac 1{N-1}} \textcolor{black}{e^{\alpha_N s^{\frac{N}{N-1}}}}, &&  s> s_0;}
\item $f(s)$ is monotone increasing, hence $F(s)=\int_0^sf(\tau)d\tau \leq sf(s)$, while the quantity $\frac{F(s)}{f(s)}$ is well defined and vanishes only at $s=0$. Furthermore,
	\begin{equation}\label{F/f}
	\frac{d}{d s}\left(\frac{F(s)}{f(s)}\right)=\frac{f^2(s)-F(s)f'(s)}{f^2(s)}\leq \frac 2N -\delta
	\end{equation}
	which implies $F(s)\leq (\frac 2N-\delta) sf(s)$;
	
	\item $(f_3)$ implies a fine lower bound on the quotient  $\displaystyle\frac{Ff'}{f^2}$, as  $s\to +\infty$. Indeed, for any $\varepsilon>0$ there exists $s_\varepsilon>0$ such that
	\syslab[ \label{estF/f}
	\frac{Ff'}{f^2}(s)\geq ]
	{ &	\left(\frac 2N+\delta\right) s, & & s\leq s_\varepsilon\\
	& (1-\varepsilon)s, && s>s_\varepsilon ;
		}
	
	\item $(f_4)$ is in the spirit of the de Figueiredo--Miyagaki--Ruf condition \cite{dFMR} and turns out to be a suitable compactness condition in this context. The role of condition $(f_4)$ will be detailed in Section \ref{subsecMP}.
			\end{itemize}
	
\noindent Examples of functions $F(s)$ satisfying our set of assumptions are given below:
	\begin{equation*}
	F(s) =\begin{cases} s^q, s\leq s_0\\
			e^{\alpha s^{\frac{N}{N-1}}}, s>s_0
			\end{cases} ,\ \forall \, q>N;
\end{equation*}
\begin{equation*}
F(s)=s^pe^{\alpha s^{\frac{N}{N-1}}},\ \forall \, p>N;
\end{equation*}
\begin{equation*}
	F(s)= \begin{cases}s^q, \quad s\leq s_0\\
			\alpha s^{p}e^{\beta s^{\frac{N}{N-1}}},\quad  s>s_0
			\end{cases}
	 \end{equation*}
		for $q\geq N$, $p> 1$ and suitable constants $\alpha,\beta>0$.
	
\medskip

\noindent Let us now introduce some basic notation:
	\[ \|u\|_N:= \|u\|_{L^N(\Rn)}\]
	and
\[ \| u\|:= \|u\|_{W^{1,N}(\Rn)} = \left(\|\nabla u\|^N_N + \| u\|^N_N\right)^{\frac{1}{N}}.\]
Let $w(x):=\log (e+|x|)$ and define the  weighted Sobolev space  $W^{1,N} L^q_{w}(\Rn)$ as the completion of smooth compactly supported functions with respect to the norm
\begin{multline*} \|u\|_{q,w}^N = \|\nabla u\|_N^N + \|u\|_{L^q(w dx)}^{N} \\
= \int_{\R^N}|\nabla u|^N \, dx+\left( \int_{\R^N}|u|^q\log(e+|x|)\,dx\right)^{N/q} .\end{multline*}
When $q=N$, for simplicity we denote $W_{w}^{1,N}(\R^N):=W^{1,N} L^N_{w}(\Rn)$ and
\[
\|u\|_{w}^N : =\|\nabla u\|_N^N+\|u\|_{L^N(w dx)}^N=\int_{\R^N}|\nabla u|^Ndx+\int_{\R^N}|u|^N\log(e+|x|)\, dx.
\]
Let us set
\[  \|u\|_V := \left(\int_{\Rn} |\nabla u|^N + V |u|^N \, dx \right)^{\frac{1}N},
\]
and we use $W^{1,N}_V(\Rn)$ to denote the set of all functions with bounded $\|\cdot\|_V$ norm. Let us also set
 $w_0(x):=\log (1+|x|)$, and
\[ \|u\|_{q,V,w_0}^N := \| u\|_V^N + \|u\|^N_{L^q(w_0 dx)} = \|u\|_V^N+\left(\int_{\R^N}|u|^q\log(1+|x|)dx\right)^{\frac{N}q},\]
and consider $W^{1,N}_V L^q_{w_0}(\Rn)$  as the completion of smooth compactly supported functions with respect to the norm $\|\cdot\|_{q,V,{w_0}}$.

	\medskip

\noindent The proper function space setting in which the energy and the variational framework turns out to be well defined, will be a consequence of the following weighted version of the Pohozaev--Trudinger inequality, which we state here for simplicity in the case $q=N$ (see Section \ref{log_p_t} for the case $q>N$):

\begin{theorem}\label{thm_A1}
The weighted Sobolev space $W_{w}^{1,N}(\R^N)$ embeds into the weighted Orlicz space $L_{\phi_N}(\R^N, \log(e+|x|)dx)$ where
\[ \phi_N(t)=e^{t}-\sum_{j=0}^{N-2} \frac{t^j}{j!}.\] More precisely, the following holds
\begin{equation}\label{wT}
		\int_{\R^N} \phi_N\left(\alpha |u|^{\frac{N}{N-1}}\right)
\log(e+|x|)dx<\infty\ ,
\end{equation}
for any $u\in W^{1,N}_{w}(\R^N)$ and any $\alpha >0$.
Moreover, the following uniform bound holds
\begin{equation}\label{wM}
\sup_{\|u\|_{w}^N\leq 1} \int_{\R^N}\phi_N\left(\alpha_N \left(\frac{N}{N+1}\right)^{1/(N-1)} |u|^{\frac{N}{N-1}}\right)\log(e+|x|)dx <+\infty\ ,
\end{equation}
where $\alpha_N=N\omega_{N-1}^{\frac 1{N-1}}$ is the sharp Moser exponent, and $\omega_{N-1}$ is the $(N-1)-$dimensional surface of the unit sphere in  $\Rn$.
\end{theorem}

\medskip

\noindent Inequality \eqref{wM} and its version in the case $q>N$ (Theorem \ref{thm_A2}), turn out to be key ingredients to obtain the following result:
	\begin{theorem}\label{thm1}
		Suppose the nonlinearity $f$ satisfies $(f_1)$--$(f_4)$ and that the potential $V$ enjoys $(V)$. Then, problem \eqref{main} possesses a nontrivial mountain pass solution which has finite energy in the weighted Sobolev space $W^{1,N}_V L^q_{w_0}(\Rn)$.
	\end{theorem}
	
\medskip
	
\subsection*{Overview} In the next section we collect some preliminary material. Special attention is devoted to discuss equivalence between \eqref{main} and \eqref{gen_sys}. This is a quite delicate matter and still with some shadows which prevent to obtain optimal results. We are motivated by a very recent debate on this topic, towards a better understanding of the higher order fractional context.  

\noindent In Section \ref{log_p_t}, we prove some fundamental results which from one side extend classical embeddings from Functional Analysis, due independently to Pohozaev and Trudigner in late sixties, on the other side provide a new tool in the `variational toolbox' to prove existence results by variational methods; we are confident these results will be useful in other situations. Here we extend in a non trivial fashion to any dimension, previous results obtained in \cite{castar} in dimension two and then applied to prove the existence of finite energy solutions to Schr\"odinger--Newton systems by variational techniques. 

\noindent In Section \ref{var_frame}, we exploit the abstract results of Section \ref{log_p_t} to provide a suitable variational framework in which we can prove the existence of a mountain pass solution to \eqref{main}. Due to the presence of exponential growth in the nonlinearity and of a sign-changing logarithmic kernel in the nonlocal part of the equation, here even the most standard variational steps become somehow delicate. We take care of stressing differences with the two dimensional case, in particular passing from semilinear, in dimension $N=2$, to quasilinear nonlocal Schr\"odinger equations in higher dimensions $N\geq 3$, where some new ideas and efforts are needed. 

\section{On the equivalence between nonlocal equations and higher order fractional systems}\label{heu}
\noindent Here we discuss the equivalence between the Choquard type equation \eqref{main} and the higher order fractional system  \eqref{gen_sys}.
Formally, if in \eqref{main} we set
\[ \Phi_u:= I_N * F(u),\]
then the function $u$ solves the equation
\bgs{ \label{choq} -\Delta_N u + V |u|^{N-2} u = \Phi_u f(u), \qquad \mbox{in} \quad \Rn}
and moreover, $ \Phi_u $ is the unique solution in $\Rn$ to the following fractional equation
$$(-\Delta)^{\frac{N}2} \phi= F(u)\ .$$
However, this argument is affected by the notion of solution we deal with, this is somehow a delicate matter and not yet completely understood. Having in mind the commitment to make precise in the sequel what we mean by solution, we have the following, and for the moment heuristic 
\begin{proposition} \label{heu}
Let $u\in W^{1,N}_VL^1_{w_0}(\Rn)$ be a solution of \eqref{main}. Then  $u$ is a solution to
\eqlab{\label{sys}-\Delta_N u + V |u|^{N-2} u &= \phi f(u) &\mbox{in } & \Rn,}
where $\phi$ is the unique solution to
\bgs{
	(-\Delta)^{\frac{N}2} \phi &= F(u) &\mbox{in } & \Rn.
}
\end{proposition}

\noindent Let us begin by recalling the definition in the distributional sense of the fractional Laplacian of any order. Set for $s>0$,
\[
L_s(\Rn):=\Big\{ u\in L^1_{\scriptsize{\mbox{loc}}}(\Rn) \; \Big| \; \int_{\Rn}\frac{|u(x)|}{1+|x|^{N+2s}} \, dx <\infty\Big\} ,
\]
the operator $(-\Delta)^s u$ is defined for all $u\in L_s(\Rn)$ via duality, as
\eqlab{ \label{ipo3}
	\langle  (-\Delta)^s u, \varphi \rangle =\int_{\Rn}  u \, (-\Delta)^s \varphi  \, dx, \qquad \forall \varphi \in \mathcal S(\Rn),
}
where
\[
(-\Delta)^s \varphi= \mathcal F^{-1}\left( |\xi|^{2s} \mathcal F \varphi(\xi) \right) , \qquad \forall \varphi \in \mathcal S(\Rn)
\]
denoting by $\mathcal F$ the Fourier transform and where $\mathcal S(\Rn)$ denotes the Schwartz space of rapidly decreasing functions.
We remark that the right hand side of \eqref{ipo3} is well defined, thanks to the fact that for $\varphi\in \mathcal S(\Rn)$,
\bgs{ |(-\Delta)^s \varphi(x)|\leq \frac{C}{|x|^{N+2s}},}
see e.g. \cite[Proposition 2.1]{hyder_structure}.

\noindent Let us consider the fractional Poisson's equation
\begin{equation}\label{ipo4}
(-\Delta)^su=f \quad  \hbox{ in }\  \Rn,  \hbox{ with }\ 0<s\leq \frac N2
\end{equation}
If $s\in(0,1)$, the setting  and the representation formulas for solutions of this equation are settled in literature, among which recent studies carried out in \cite{Bucur,Abaty}. The case $s>1$ has been considered in very recent papers  \cite{Abat,stinga,BGV}, whereas a general approach, based on the notion of distributional solution dates back to classical works \cite{Stein,Landkof,samkilmar}, see also \cite{galdi}.

 \begin{definition}\label{ipo12}
 	Given $f\in \mathcal S'(\Rn)$, we say that $u\in L_{\frac{N}2}(\Rn)$ is a {\it{ distributional solution}} of \eqref{ipo4}  if
 	\[
 	\int_{\Rn} u  (-\Delta)^{\frac{N}2}\varphi \, dx = \langle f, \varphi \rangle
 	\]
 	for all $\varphi \in \mathcal S(\Rn)$.
 \end{definition}

 \noindent It is well known that, if $s<\frac N2$, the distributional solution of \eqref{ipo4} is given by
   $$
   u(x)=\mathcal F^{-1}\left(\frac{1}{|\xi|^{2s}}\mathcal F f(\xi)\right)(x)
   $$
which is realized also by convolution with the Newtonian potential
$$
u(x)=(I_{2s}\ast f)(x), \qquad \hbox{ where } \ I_{2s}(x)=\frac{1}{\gamma_{N,s}}|x|^{2s-N}
$$
 for some $\gamma_{N,s}>0$, (see \cite[Chapter 5]{samkilmar}).
 Note that the Newtonian potential, in the case $2s=N$, is given by
 $$
 I_N(x)=\frac{1}{\gamma_N}\log\frac 1{|x|}=\mathcal F |\xi|^{-N}(x),
 $$
 Nevertheless, when $2s=N$ it is not possible to define the solution to \eqref{ipo4} by Fourier transform  in $\mathcal S(\Rn)$ in general, due to the singularity of $|\xi|^{-N}$ in zero.  However, various assumptions on $f$, which improve the regularity of its Fourier transform, allow to  recover  the notion of distributional solution, for instance, the assumption
 $$
 \mathcal F f(0)=0,  \quad \hbox{ that is } \quad \int_{\Rn}f(x)dx=0.
 $$
 \noindent The notion of Fourier transform has to be settled in a suitable framework, such as Lizorkin's spaces, defined as the subspace of Schwartz functions which are orthogonal to polynomials, namely :
 \[
 \Phi=\left\{ \varphi(x) \, \big|  \varphi\in \mathcal S(\Rn), \, \int x^j\varphi(x)=0, \,\, \forall \,|j|\in \mathbb N_0 \right\}.
\]
 Again, the convolution with a $\log$-kernel does not yield enough $L^1_{loc}$-regularity to provide a notion of distributional solution in the general context of $\mathcal S(\Rn)$. See \cite[Chapter 5]{samkilmar} for more details.

 \noindent Notice that when $N$ is even, $(-\Delta)^{N/2} $ is an integer order operator, so its fundamental solution in $\Rn$ is known, see e.g. \cite[Proposition 22]{marty09}.

\noindent When $N$ is odd, the fractional case, an alternative  approach to circumvent the loss of regularity in the borderline case $2s=N$ is given in \cite{hyder_structure} (see also \cite{martinazzi1,martinazzi2,hyderentire}). We next recall some ideas from \cite{hyder_structure}. The argument is simple, the $N/2$-Laplacian can be seen as the composition of the  $1/2$-Laplacian with the Laplacian of integer order $(N-1)/2$.  The following proposition ensures that the logarithmic potential $I_N$ is the fundamental solution of the $N/2$-Laplacian in this sense.
\begin{proposition}\cite[Lemma A.2]{hyder_structure}
	Let $N\geq 3$ be an odd integer number and define
	\[
	\Phi(x) :=(-\Delta)^{\frac{N-1}2} I_N(x) = \frac{c_n}{|x|^{N-1}}.
	\]
	Then $\Phi$ is the fundamental solution of $(-\Delta)^{\frac12} $ in $\Rn$, in the sense that for all $f\in L^1(\Rn)$ it holds that $\Phi*f \in L_{\frac{1}2}(\Rn)$ and that
	\[
	\langle (-\Delta)^{\frac{1}2} (\Phi*f), \varphi\rangle := \int_{\Rn} (\Phi*f)(x)	(-\Delta)^{\frac{1}2}  \varphi(x) \, dx = \int_{\Rn} f\varphi \, dx,
	\]
	for all $\varphi \in \mathcal S(\Rn)$.
\end{proposition}

\noindent However, it is not straightforward from here that $I_N$ is the fundamental solution of $(-\Delta)^{N/2}$ in the sense of Definition \ref{ipo12}. Actually, by interpreting the $N/2$-Laplacian (when $N$ is odd) as the composition of the  Laplacian of (integer) order $(N-1)/2$ and the $1/2$-Laplacian, one has that Definition \ref{ipo12} turns out to be equivalent to the following 

\begin{definition}{\cite[Definition 1.1]{hyder_structure}}
	\label{ipo11}
	Given $f\in \mathcal S'(\Rn)$, we say $u$ is a solution of \eqref{ipo4} if
	\[ u\in W^{N-1,1}_{\loc}(\Rn), \qquad  \Delta^{\frac{N-1}2} u \in L_{\frac12} (\Rn),\]
	and
	\[ \int_{\Rn} (-\Delta)^{\frac{N-1}2} u (-\Delta)^{\frac12} \varphi \, dx = \langle f, \varphi \rangle
	\]
	for all $\varphi \in \mathcal S(\Rn)$.
\end{definition}

\noindent Indeed, we have 
\begin{proposition}{\cite[Proposition 2.6]{hyder_structure}}
	Let $f\in L^1(\Rn)$. Then $u$ is a solution of \eqref{ipo4} in the sense of Definition \ref{ipo11} if and only if $u$ is a solution in the sense of Definition \ref{ipo12}.
\end{proposition}

\noindent As already pointed out, the convolution $I_N*f$ itself does not provide in general a distributional solution of \eqref{ipo4}. Nevertheless, a suitable modification of the logarithmic potential $I_N$, from one side yields enough $L^1_{\loc}$-regularity, on the other side it is the fundamental solution of the $N/2$-Laplacian, in the distributional sense of Definition \ref{ipo12}.

\begin{lemma}{\cite[Lemma 2.3]{hyder_structure}}\label{kko}
	Let $f\in L^1(\Rn)$, and for all $x\in \Rn$,
	\[ \tilde v(x) := \frac{1}{\gamma_N} \int_{\Rn} \log \left( \frac{1+|y|}{|x-y|}\right) f(y) \, dy.\]
	Then $\tilde v \in W^{N-1,1}_{\loc}(\Rn)$ and
	\begin{multline*}
	\int_{\Rn} \tilde v(x) (-\Delta)^{\frac{N}2} \varphi(x) \, dx = \int_{\Rn} (-\Delta)^{\frac{N-1}2} \tilde v(x) (-\Delta)^{\frac{1}2} \varphi(x) \, dx  \\
=\int_{\Rn} f(x) \varphi(x) \, dx
\end{multline*}
	for all $\varphi \in \mathcal S(\Rn)$,
	i.e, $\tilde v$ is a distributional solution (in the sense of Definitions \ref{ipo11}, \ref{ipo12}) of \eqref{ipo4}.
\end{lemma}

\begin{lemma}{\cite[Lemma 2.4]{hyder_structure}}
	Let $u$ be a solution of \eqref{ipo4} in the sense of Definition \ref{ipo12} with $f\in L^1(\Rn)$.
	Then
	\[
	u= \tilde v + p
	,\]
	where $p$ is a polynomial of degree at most $n-1$.
\end{lemma}

\noindent As a consequence of what we have recalled here from \cite{hyder_structure}, we are now in the position to proof Proposition  \ref{heu}. 
\begin{proof}[Proof of Proposition \ref{heu}]
Let $u\in W^{1,N} L^q_{w_0}(\Rn)$ be a (weak) solution of \eqref{main}. Then $F(u) \in L^p(\Rn)$ for any $p\geq 1$, as a consequence of Theorems \ref{thm_A1} and \ref{thm_A2}. Let
\[ \tilde v(x) = \int_{\Rn} \log \left(\frac{1+|y|}{|x-y|}\right) F(u(y)) \, dy.
\]
Let us rewrite equation \eqref{main} as follows 
\begin{multline*}
  (-\Delta)_N u(x) +V(x) |u(x)|^{N-2} u(x) \\
  = \tilde v (x) f(u(x)) +  \big[(I_N*F(u)) (x) -\tilde v(x)\big] f(u(x))\ .
\end{multline*}
Set 
\[ \tilde \Phi(x):= I_N*F(u) (x) = \tilde v (x) +  \big[(I_N*F(u)) (x) -\tilde v(x)\big]  \]
and recall that 
\bgs{
	&( I_N*F(u)) (x) -\tilde v(x)  =
\\  =& \; \frac{1}{\gamma_N}\int_{\Rn} \left( \log \frac{1}{|x-y|} - \log \left(\frac{1+|y|}{|x-y|}\right)\right) F(u(y)) \, dy\\
	=&\; - \frac{1}{\gamma_N}\int_{\Rn} \log (1+|y|)F(u(y)) \, dy \ .
}
Since $u\in W^{1,N} L^q_{w_0}(\Rn)$, according to Theorem \ref{thm_A2}
\[ \int_{\Rn} \log (1+|y|)F(u(y)) \, dy <\infty,\]
hence $\kappa_N:= (I_N*F(u)) (x) -\tilde v(x)$ is a constant function. Hence
\[ \tilde \Phi(x)= \tilde v  (x)+  \kappa_N \in W^{N-1,1}_{\loc}(\Rn)\]
and according to Lemma \ref{kko}, for any $\varphi \in \mathcal S(\Rn)$ we  have
	\bgs{
		\int_{\Rn} \tilde \Phi (-\Delta)^{\frac{N}2} \varphi \, dx=&\; \int_{\Rn} (-\Delta)^{\frac{N-1}2}\tilde  \Phi (-\Delta)^{\frac12} \varphi \, dx
		\\
		= &\; \int_{\Rn} (-\Delta)^{\frac{N-1}2} (\tilde v +\kappa_N) (-\Delta)^{\frac12} \varphi \, dx
		\\
		=&\; \int_{\Rn} (-\Delta)^{\frac{N-1}2} \tilde v (-\Delta)^{\frac12} \varphi \, dx
		\\
		=&\; \int_{\Rn} F(u) \varphi \, dx .}
		Indeed, we point out that since $N$ is odd, $ (-\Delta)^{\frac{N-1}2}$ is an integer order operator and hence $ (-\Delta)^{\frac{N-1}2} \kappa_N=0$.  Therefore,
$\tilde \Phi$ is a distributional solution of $(-\Delta)^{\frac{N}2} \phi = F(u)$, in the sense of Definition \ref{ipo12}.
\end{proof}
\begin{remark}
  Notice that 
  $$
I_N*F(u)\in L^1_{\loc}(\Rn)
$$
since
\begin{multline*}
\int_{\Omega}|I_N*F(u)|dx\leq \int_{\Omega\times\Rn}|\log|x-y||F(u)dxdy\leq\\
\leq C_{\mu}\int_{\Omega\times\Rn} \left[\frac{1}{|x-y|^\mu}+\log(1+|x|)+\log(1+|y|)\right]F(u(y))dxdy<+\infty
\end{multline*}
where $\mu>0$. Boundedness follows by the Hardy-Littlewood-Sobolev inequality and from Theorems \ref{thm_A1} and \ref{thm_A2}. Similarly one also has 
$$
I_N*F(u)\in L_{\frac N2}(\Rn)
$$
and that it is a distributional solution of
$$
(-\Delta)^{N/2}\phi=F(u)
.$$
\end{remark}

\noindent We conclude this preliminary section by recalling two classical versions of the Hardy--Littlewood--Sobolev inequality which will be used later on:
\begin{proposition}[HLS inequality]\label{HLS}
	Let $s, r>1$ and $0<\mu<N$ with $1/s+\mu/N+1/r=2$, $f\in
	L^s(\R^N)$ and $g\in L^r(\R^N)$. There exists a constant
	$C(s,N,\mu,r)$, independent of $f,h$, such that
	$$
	\int_{\R^N}\left[\frac{1}{|x|^{\mu}}\ast f(x)\right]g(x)\leq
	C(s,N,\mu,r) \|f\|_s\|g\|_r.
	$$
\end{proposition}

\begin{proposition}[Logarithmic HLS
	inequality]\label{LHLS}Let $f,g$ be two nonnegative
	functions belonging to $L\ln L(\R^N)$, such that $\int f \log(1+|x|)<\infty, \int g \log(1+|x|)<\infty $ and $\|f\|_1=\|g\|_1=1$. There exists a constant $C_N$,
	independent of $f,g$, such that
	\begin{equation*}
	2N\int_{\R^N}\left[\log \frac{1}{|x|}\ast f(x)\right]g(x)\leq
	C_N+\int_{\mathbb R^N}f\log f
	dx+\int_{\mathbb R^N}g\log g dx\ .
	\end{equation*}
\end{proposition}

\section{A \texorpdfstring{$\log$}{}-weighted Pohozaev--Trudinger type inequality in \texorpdfstring{$\R^N$}{}}\label{log_p_t}
\noindent In this section, we prove a Pohozaev--Trudinger type inequality in the whole $\R^N$, with a logarithmic weight which appears just in the mass part of the energy. The prototype weight is \[ w=\log(e+|x|),\]  which plays a role only as $|x|\to +\infty$.  The main result of this section is a quite involved extension of \cite[Theorem 3.1]{castar}, where the two dimensional case was considered. We begin with the case $q=N$ whence the case $q>N$ will be covered in Theorem \ref{thm_A2}.

\begin{proof}[Proof of Theorem \ref{thm_A1}]
We perform a change of variables, by using hypershperical coordinates in $\R^N$, to pass from $W_{w_0}^{1,N}(\R^N)$ to $W^{1,N}(\R^N)$ as follows

$$
x=
\begin{cases}
 x_1=|x|\sin \theta_1 \sin \theta_2\dots \sin \theta_{N-2}\sin \theta_{N-1}\\
 x_2=|x|\sin \theta_1 \sin \theta_2\dots \sin \theta_{N-2}\cos \theta_{N-1}\\
 x_3=|x|\sin \theta_1 \sin \theta_2\dots \cos \theta_{N-2}\\
 \dots\\
 x_N=|x|\cos \theta_1,
\end{cases}
$$
where $\theta_1, \dots \theta_{N-2}\in  [0,\pi]$, whence $\theta_{N-1}\in [0, 2\pi)$, $|x|^2=x_1^2+\dots +x_N^2$.
By acting only on the radial component of a point in $\Rn$, set 
$$
T(|x|)=|y|, \ \ \frac y{|y|}=\frac x{|x|}, \ \ |y|=|x|\sqrt[N]{\log(e+|x|)}.
$$
We set $r=|x|$ and $s=|y|$, hence $s=T(r)=r\sqrt[N]{\log(e +r)}$. We obtain
$$
T'(r)=\frac{N\log(e+r)+\frac{r}{e+r}}{N[\log(e+r)]^{\frac{N-1}{N}}}>0, \ \ T(0)=0, \ \lim_{r\to +\infty} T(r)=+\infty
$$
and thus $T$ is invertible on $\R^N$ (though the inverse map is not explicitly known). 
\noindent Set
\[
v(y):=u(x)
\]
or, equivalently
\begin{multline*}
  \ u(r\sin \theta_1\dots\sin \theta_{N-1},\dots, r\cos \theta_1) \\
v\left(T(r)\sin \theta_1\dots\sin \theta_{N-1},\dots, T(r)\cos \theta_1\right)\ .
\end{multline*}
Then, denoting $\theta=(\theta_1, \dots, \theta_{N-1})$ and
\begin{eqnarray*}
&&w(r, \theta
):=u(r\sin \theta_1\dots\sin \theta_{N-1},\dots, r\cos \theta_1)\\
&& \widetilde w(s, \theta
):=v\left(s\sin \theta_1\dots\sin \theta_{N-1},\dots, s\cos \theta_1 \right), \\
&& w(r, \theta
)= \widetilde w\left( T(r), \theta\right
),
\end{eqnarray*}
we compute
\begin{eqnarray*}
&&w_r(r, \theta
)= \widetilde w_s\left( T(r), \theta
\right)T'(r), \\
&&w_{\theta_i}\left( T(r), \theta
\right)=\widetilde w_{\theta_i}\left( T(r),  \theta
\right) \ \ \  i=1, \dots, N-1.
\end{eqnarray*}
Therefore 
\bgs{
& \int_{\R^N}|\nabla v|^N\, dy\\
			=&\; \int_{0}^{2\pi} \int_0^\pi \sin \theta_{N-2}\dots \int_0^\pi \sin^{N-2}{\theta_1}\int_0^{+\infty}\left[\widetilde w_s^2(s,\theta)+\frac{\widetilde w^2_{\theta_1}(s,\theta)}{s^2}+\dots\right.
	\\
		&\; \left.+\frac{\widetilde w^2_{\theta_{N-1}}(s,\theta)}{s^2\sin^2\theta_1\dots \sin^2 \theta_{N-2}}\right]^{\frac N2}s^{N-1}ds\, d\theta_1 \dots  d\theta_{N-2} \, d\theta_{N-1}\\
		=&\;\int_{0}^{2\pi}\int_0^\pi \sin \theta_{N-2}\dots \int_0^\pi \sin^{N-2}{\theta_1}
			\int_0^{+\infty}\left[\widetilde w_s^2(T(r),\theta )+\frac{\widetilde w^2_{\theta_1}(T(r),\theta)}{T^2(r)}+\dots\right.
	\\
		&\;\left.+\frac{\widetilde w^2_{\theta_{N-1}}(T(r), \theta)}{T^2(r)\sin^2\theta_1\dots \sin^2 \theta_{N-2}}\right]^{\frac N2}T^{N-1}(r)T'(r)dr \, d\theta_1 \dots  d\theta_{N-2} \, d\theta_{N-1}
	\\
	=&\;\int_{0}^{2\pi} \int_0^\pi \sin \theta_{N-2} \dots \int_0^\pi \sin^{N-2}{\theta_1}\int_0^{+\infty}\left[\frac{ w_r^2(r,\theta
)}{[T'(r)]^2}+\frac{ w^2_{\theta_1}(r,\theta)}{r^2}\frac{r^2}{T^2(r)}+\dots\right.
	\\
		&\;\left.+\frac{w^2_{\theta_{N-1}}(r, \theta)}{r^2}\frac{r^2}{T^2(r)\sin^2\theta_1\dots \sin^2 \theta_{N-2}}\right]^{\frac N2}T^{N-1}T'dr d\theta_1 \dots  d\theta_{N-2}d\theta_{N-1}.
}
Now, since
\[
\frac{1}{[T'(r)]^2}=\frac{\left[\log(e+r)\right]^{\frac{2(N-1)}N }}{\left[\log(e+r)+\frac{r}{N(e+r)}\right]^2}, \ \ \frac{r^2}{T^2(r)}=\frac{1}{\left[\log(e+r)\right]^{2/N}}
\]
we get 
\[
\frac N{N+1} \frac{r^2}{T^2(r)}< \frac{1}{[T'(r)]^2}< \frac{r^2}{T^2(r)}\ .
\]
Thus we have 
\begin{multline*}
\left( \frac N{N+1}\right)^{\frac N2} \int_{0}^{2\pi}\int_0^\pi \sin \theta_{N-2}\dots \int_0^\pi \sin^{N-2}{\theta_1}\\
\cdot\, \int_0^{+\infty}\left[w_r^2(r,\theta)+\frac{ w^2_{\theta_1}(r,\theta)}{r^2}+\dots\right.\\
\left.+\frac{w^2_{\theta_{N-1}}(r, \theta)}{r^2\sin^2\theta_1\dots \sin^2 \theta_{N-2} }\right]^{\frac N2}\frac{r^NT'(r)}{T(r)}dr\, d\theta_1 \dots  d\theta_{N-2} \, d\theta_{N-1}
	\\
	\leq  \int_{\R^N}|\nabla v|^N dy_1\dots dy_N
	\\
	\leq  \int_{0}^{2\pi}\int_0^\pi \sin \theta_{N-2}\dots \int_0^\pi \sin^{N-2}{\theta_1}\int_0^{+\infty}\left[w_r^2(r,\theta)+\frac{ w^2_{\theta_1}(r,\theta)}{r^2}+\dots\right.
	\\
	 \left. +\frac{w^2_{\theta_{N-1}}(r, \theta)}{r^2\sin^2\theta_1\dots \sin^2 \theta_{N-2}}\right]^{\frac N2}\frac{r^NT'(r)}{T(r)}dr\, d\theta_1 \dots  d\theta_{N-2} \, d\theta_{N-1}\ .
\end{multline*}
On the one hand, from 
$$
\frac{r^NT'(r)}{T(r)}= r^{N-1}\left[1+ \frac r{N(e+r)\log(e+r)}\right]$$
one has 
\begin{equation}\label{uno}
r^{N-1}<\frac{r^NT'(r)}{T(r)}<\frac{N+1}N r^{N-1}
\end{equation}
and then 
\begin{equation}\label{Tgrad}
\left(\frac N{N+1}\right)^{\frac N2} \int_{\R^N}|\nabla u|^N \, dx<\int_{\R^N}|\nabla v|^N\, dy<\frac{N+1}N\int_{\R^N}|\nabla u|^N \,dx.
\end{equation}
On the other hand,
\bgs{
	&\int_{\R^N}|v|^N \, dy \\
	=&\;
	 \int_{0}^{2\pi} \int_0^\pi \sin \theta_{N-2}\dots \int_0^\pi \sin^{N-2}{\theta_1} 				\int_0^{+\infty}|\widetilde w(s, \theta)|^Ns^{N-1} ds
	d\theta_1\dots d\theta_{N-1}
	\\
	=  &\;\int_{0}^{2\pi} \int_0^\pi \sin \theta_{N-2}\dots \int_0^\pi \sin^{N-2}{\theta_1} \int_0^{+\infty} |\widetilde w(T(r), \theta)|^NT'T^{N-1}dr\dots d\theta_{N-1}
	\\
	=&\;
	\int_{0}^{2\pi}\int_0^\pi \sin \theta_{N-2}\dots \int_0^\pi \sin^{N-2}{\theta_1}\int_0^{+\infty} |w(r, \theta)|^NT'T^{N-1}dr\, d\theta_1 \dots d\theta_{N-1}.
}
Notice that
\eqlab{ \label{fu1}
T'(r)T^{N-1}(r)= &\; r^{N-1}\left[\log(e+r)+\frac{r}{N(e+r)}\right]\\
=&\;  r^{N-1}\log(e+r)\left[1+\frac{r}{N(e+r)\log(e+r)}\right]\\
=&\frac{r^NT'(r)}{T(r)}\log(e+r)
}
and hence 
\begin{multline*}
\int_{\R^N}|v|^Ndy = \int_{0}^{2\pi} \int_0^\pi \sin \theta_{N-2}\dots \\
\dots \int_0^\pi \sin^{N-2}{\theta_1} \int_0^{+\infty} |w(r,\theta)|^N \frac{r^NT'}{T}\log(e+r) dr \dots d\theta_{N-1} \ . 
\end{multline*}
By \eqref{uno},
\[
\int_{\R^N}|u|^N\log(e+|x|)dx<\int_{\R^N}|v|^N \, dy<\frac{N+1}{N}\int_{\R^N}|u|^N\log(e+|x|)dx\ .
\]
Finally,
\begin{equation}
\left(\frac{N}{N+1}\right)^{\frac N2}\|u\|_{w}^N<\|v\|^N<\frac{N+1}{N}\|u\|_{w}^N\ .
\label{normuv}
\end{equation}
We have hence proved that the map
\begin{eqnarray*}
 \mathfrak{T}: W^{1,N}_{w}(\mathbb R^N)&\to& W^{1,N}_0(\mathbb R^N)\\
 u&\mapsto& v
\end{eqnarray*}
is invertible, continuous and with continuous inverse. Then, similarly as above, 
\begin{multline*}
 \int_{\R^N} \phi_N\left( \alpha |u(x)|^{\frac{N}{N-1}}\right)\log(e+|x|)dx
 	\\
 =\int_{0}^{2\pi}\dots \int_0^\pi \sin^{N-2}{\theta_1}\\
 \int_0^{+\infty}\phi_N\left( \alpha |w(r, \theta)|^{\frac{N}{N-1}}\right)\log(e+r)r^{N-1}dr\,d\theta_1 \dots  d\theta_{N-1}
 \\
 =\int_{0}^{2\pi}\dots \int_0^\pi \sin^{N-2}{\theta_1}\\
 \int_0^{+\infty}\phi_N\left( \alpha |\widetilde{w}(T(r), \theta )|^{\frac{N}{N-1}}\right)\frac{\log(e+r)r^{N-1}}{T'(r)T(r)^{N-1}}T'(r)T(r)^{N-1}dr \,d\theta_1 \dots d\theta_{N-1}
 \\
  \leq\int_0^{2\pi}\dots \int_0^\pi \sin^{N-2}{\theta_1}\int_0^{+\infty}\phi_N\left( \alpha |\widetilde{w}(\rho,\theta)|^{\frac{N}{N-1}}\right)\rho^{N-1} d\rho \,d\theta_1 \dots  d\theta_{N-1}
  \\
  =  \int_{\R^N}\phi_N\left(\alpha |v|^{\frac{N}{N-1}}\right)dx <+\infty
\end{multline*}
 for any $\alpha>0$, where we have used \eqref{uno}, \eqref{fu1}, and  \cite[Theorem 1.1]{LR} in the last line.
The uniform bound \eqref{wM} follows directly from  \eqref{normuv}. Indeed, for any $u\in W^{1,N}_{w}(\Rn)$ and $\alpha \leq \alpha_N \left(\frac{N}{N+1}\right)^{1/(N-1)}$ one has
\begin{multline*}
 \int_{\R^N} \phi_N\left( \alpha \left(\frac{|u|}{\|u\|_{w}}\right)^{\frac{N}{N-1}}\right)\log(e+|x|)dx
 	\\
 =\int_0^{2\pi}\dots \int_0^\pi \sin^{N-2}{\theta_1}\\
 \int_0^{+\infty}\phi_N\left( \alpha \left(\frac{|w(r, \theta)|}{\|u\|_{w}}\right)^{\frac N{N-1}}\right)\log(e+r)r^{N-1}dr\,d\theta_1 \dots  d\theta_{N-1}
 \\
 = \int_0^{2\pi}\dots \int_0^\pi \sin^{N-2}{\theta_1}\\
 \int_0^{+\infty}\hspace{-.5cm}\phi_N\left( \alpha \left(\frac{|\widetilde w(T(r),\theta)|}{\|u\|_{w}}\right)^{\frac N{N-1}}\right)\frac{\log(e+r)r^{N-1}}{T'(r)T(r)^{N-1}} T'(r)T(r)^{N-1}dr \dots  d\theta_{N-1}
 \\
  \leq \; \int_0^{2\pi}\dots \int_0^\pi \sin^{N-2}{\theta_1}\\
  \int_0^{+\infty}\phi_N\left( \alpha_N \left(\frac{|\widetilde{w}(\rho, \theta)|}{\|v\|}\right)^{\frac N{N-1}}\right)\rho^{N-1} d\rho \,d\theta_1 \dots  d\theta_{N-1}
  \\
  =  \int_{\R^N}\phi_N\left( \alpha_N \left(\frac{|v|}{\|v\|}\right)^{\frac N{N-1}}\right)dx <C
\end{multline*}
by \eqref{normuv} and using  \cite[Theorem 1.1]{LR} for the last inequality.
\end{proof}
\noindent As a byproduct of this embedding result one has the continuity of a weighted Pohozaev--Trudinger functional on $W^{1,N}_{w}(\R^N)$, namely we have the following 
\begin{corollary}\label{cor-wcontinuity}
For any $\alpha >0$, the functional
\begin{equation*}
u \longmapsto  \int_{\R^N}\phi_N\left(\alpha |u|^{\frac N{N-1}}\right)\log(e+|x|)\, dx
\end{equation*}
is continuous on $W^{1,N}_{w}(\R^N)$.
\end{corollary}
\begin{remark}\label{remA3}
The value $\alpha_N \left(\frac{N}{N+1}\right)^{1/(N-1)}$ in \eqref{wM} is not sharp and we conjecture that the sharp value is $\alpha_N$ as in the Moser case \cite{M}, though it is somehow delicate and still out of reach.
\end{remark}

\noindent Next we consider the case in which the asymptotic growth of the nonlinearity near zero is a power $q>N$. We prove the following 

\begin{theorem}\label{thm_A2} 
Let $f\colon \R \to [0,+\infty)$ satisfying $(f_1)$ and let $q>N$. Then, the space $W^{1,N}L^q_{w}(\Rn)$ embeds into the weighted Orlicz space $L_F(\Rn, \log(e+|x|)dx)$. More precisely,
\[\int_{\Rn} F(\alpha |u|) \log(e+|x|) \, dx<\infty, \quad \forall u\in W^{1,N} L^q_{w}(\Rn), \, \forall \alpha>0 \ . \]
Moreover, for any $\alpha < \frac Nq\left(\frac{N}{N+1}\right)^{1/N}$ the following uniform bound holds 
\begin{equation}\label{wMq}
\sup_{\|u\|_{q,w}^N\leq 1} \int_{\R^N} F\left(\alpha |u|\right)\log(e+|x|)dx <+\infty\ .
\end{equation}
\end{theorem}

\begin{proof}For $u\in W^{1,N}L^q_{w}(\Rn)$ set 
	\sys[v:=]{ & |u|^{\frac{q}N}, && |u|<1,\\
				& |u|, && |u|\geq 1,}
	which belongs to 
	$W^{1,N}_{w}(\Rn)$. 			
		Indeed,
		\begin{multline*}
			\|v\|^N_{w} = \int_{\Rn} |\nabla v|^N \, dx + \int_{\Rn} |v|^N \log(e+|x|)\, dx
			\\
			 \leq \left(\frac{q}N\right)^N \int_{\Rn} |\nabla u|^N \, dx + \int_{\Rn} |u|^q \log (e+|x|)\, dx \leq \left(\frac{q}N\right)^N  \|u\|_{q,w}^N.
			 \end{multline*} 	
	\noindent Now recall from \eqref{f11} that for any $\|u\|_{q,w}\leq 1$ one has:\\
	
\noindent 	$ \bullet$ if $p\leq \frac 1{N-1}$,
	  	\begin{multline*}
	\int_{\Rn} F(\alpha |u|) \log(e+|x|) \, dx \\
\leq 	C \left( \int_{\Rn} \alpha^q |u|^q \log(e+|x|) \, dx\right.\\ \left.+ \alpha^{p-\frac{1}{N-1}} \int_{ \{ |u|>1\} } e^{\alpha_N (\alpha |u|)^{\frac{N}{N-1}}} \log(e+|x|)\, dx \right)
\\
=C \left( \int_{\Rn} \alpha^q |u|^q \log(e+|x|) \, dx \right.
\\\left. + \alpha^{p-\frac{1}{N-1}} \int_{ \{ |v|>1\} } e^{\alpha_N(\alpha\|v\|_{w})^{\frac{N}{N-1}} (|v|/\|v\|_{w})^{\frac{N}{N-1}}} \log(e+|x|)\, dx \right)\\
\leq C \left(\alpha^q \|u\|_{L^q(w dx)}^q \right.\\
\left.+  \alpha^{p-\frac{1}{N-1}}  \int_{\{|v|>1\}} \phi_N \left( \alpha_N \left(\frac{N}{N+1}\right)^{\frac 1{N-1}} \left(\frac{|v|}{\|v\|_{w}}\right)^{\frac{N}{N-1}}\right) \log(e+|x|) \, dx\right) \leq C
\end{multline*} (where the last bound is independent of $u$ in the unit ball of $W^{1,N}L^q_{w}(\R^N)$);\\

\noindent $\bullet$ if $p > \frac 1{N-1}$,
\begin{multline*}
 \int_{\Rn} F(\alpha |u|) \log(e+|x|) \, dx \\
\leq 	C \left( \int_{\Rn} \alpha^q |u|^q \log(e+|x|) \, dx \right.\\
\left. + \int_{ \{ |u|>1\} } \left(\alpha|u|\right)^{p-\frac 1{N-1}} e^{\alpha_N (\alpha  |u|)^{\frac{N}{N-1}}} \log(e+|x|)\, dx \right)
\\
=C \left(  \alpha^q \|u\|_{L^q(w dx)}^q\right. 
\\ \left.+ \int_{ \{ |v|>1\} }\left(\alpha|v|\right)^{p-\frac 1{N-1}} e^{\alpha_N (\alpha\|v\|_{w})^{\frac{N}{N-1}} (|v|/\|v\|_{w})^{\frac{N}{N-1}}} \log(e+|x|)\, dx \right)\\
\leq C \left( \alpha^q \|u\|_{L^q(wdx)}^q \right.\\
\left. + \alpha^{p-\frac 1{N-1}}\int_{\{|u|>1\}} |u|^{r'(p-\frac 1{N-1})} \int_{ \{ |v|>1\} } \hspace{-.5cm}e^{r \alpha_N (\alpha\|v\|_{w})^{\frac{N}{N-1}} (|v|/\|v\|_{w})^{\frac{N}{N-1}}} \log(e+|x|)\, dx \right) 
\end{multline*}
where $r=\left(\frac{N}{N+1}\right)^{\frac 1{N-1}}\left(\frac{N}{\alpha q}\right)^{\frac{N}{N-1}}>1$, provided $\alpha <\frac{N}{q}\left(\frac{N}{N+1}\right)^{\frac 1N}$ and  $r'$ is the Young conjugate of $r$. Hence,
\begin{multline*}
\int_{\Rn} F(\alpha |u|) \log(e+|x|) \, dx \\
	\leq C \left(\alpha^q \|u\|^q_{L^q(w dx)} +\right.
\\	 \left.  \alpha^{p-\frac 1{N-1}}\int_{\{|v|>1\}} \phi_N\left( \left(\frac{N}{N+1}\right)^{\frac 1{N-1}}\alpha_N \left(\frac{|v|}{\|v\|_{w}}\right)^{\frac{N}{N-1}}\right) \log(e+|x|) \, dx \right)
		\leq C
\end{multline*}							
(where the constant $C$ does not depend on $u$ in the unit ball of $W^{1,N}L^q_{w}(\R^N)$).	
\end{proof}
\begin{remark} The analogous of Corollary \ref{cor-wcontinuity} holds also in the case $q>N$.
\end{remark}
\section{The variational framework: proof of Theorem \ref{thm1}}\label{var_frame}
\noindent The energy functional we consider is the following 
 \[
 	\mathcal I_V(u)= \frac{1}N	\|u\|_V^N- \mathcal F(u),
 \]
 with
 \eqlab{ \label{fru1} \mathcal F(u)=&\; \frac 12 \int_{\Rn} (I_N *F(u))(x) F(u(x))\, dx
 \\
 =&\;\frac 1{2\gamma_N} \int_{\Rn}\left(\log \frac 1{|\cdot|}* F(u)\right) (x) F(u(x)) \, dx
 \\
 =&\;\frac 1{2\gamma_N} \int_{\R^{N}}  \int_{\R^{N}} \log \frac 1{|x-y|} F(u(y)) F(u(x)) \, dx\,dy.}

\noindent The regularity of $\mathcal I_V$ can be proved following line by line \cite[Lemma 4.2]{castar}, namely one has 
\begin{theorem}
The energy functional $\mathcal I_V$ is of class $C^1$ on $W^{1,N}_V L^q_{w_0}(\Rn)$.
\end{theorem}

\subsection{Mountain pass geometry}\label{subsecMP}
Let us focus here on the geometry of the energy functional $\mathcal I_V$.
\begin{lemma}\label{lemMP} The energy functional $\mathcal I_V$ satisfies the following:
\begin{itemize}
\item[$(i)$] there exist $\rho, \delta_0>0$ such that $\mathcal I_V|_{S_\rho} \geq \delta_0$ for all $u \in S_\rho$  $$S_\rho:=\displaystyle  \{ u \in W^{1,N}_V L^q_{w_0}(\Rn) \, |\, \|u\|_{q,V,w_0} =\rho\};$$

\item[$(ii)$]there exists $e \in  W^{1,N}_V L^q_{w_0}(\Rn), \|e\|_{q,V,w_0}>\rho$ such that $\mathcal I_V(e) <0$.
 \end{itemize}
\end{lemma}
\begin{proof}
Throughout the proof, constants may change from line to line. Notice that from the logarithmic Hardy-Littlewood-Sobolev inequality, Proposition \ref{LHLS}, we have 
	\bgs{ \mathcal F (u) \leq &\;\| F(u)\|_{L^1(\Rn)} \bigg( C_N \| F(u)\|_{L^1(\Rn)}  + \frac{1}N \int_{\Rn} F(u) \log F(u)\, dx
	\\
	&\; - \frac{1}N \| F(u)\|_{L^1(\Rn)} \log \| F(u)\|_{L^1(\Rn)} 	\bigg).
	}

\noindent Since $\|u\|_V \leq \|u\|_{q,V,w_0} =\rho$, for $ \rho$ small, by \eqref{f11} and noting that  for any $s>s_0$ and $p>0$

\[
s^{p-\frac{1}{N-1}}\phi_N\left(\alpha_Ns^{\frac N{N-1}}\right)\leq C_{N,p} s^{2N}\phi_N\left(2\alpha_N s^{\frac N{N-1}}\right)\]
 we have,
for some $r>1$,
\begin{multline*} \|F(u)\|_{L^1(\Rn)} \leq  C \|u\|_q^q+ C\left[\int_{\Rn} |u|^{2Nr'}dx\right]^{\frac 1{r'}}\cdot\\
\cdot \left[\int_{\Rn}\phi_N\left(2\alpha_N r\|\nabla u\|_N^{\frac N{N-1}} \left(\frac{|u|}{\|\nabla u\|_N}\right)^{\frac{N}{N-1}}\right)dx\right]^{\frac 1{r}}\\
\leq  C \left(\|u\|_V^q+ \|u\|_V^{2N} \right)\leq  C \left(\|u\|_{q,V,w_0}^N+ \|u\|_{q,V,w_0}^{2N} \right)\leq C \|u\|_{q,V,w_0}^N
\end{multline*}
and in turn,
\[ \left|\| F(u)\|_{L^1(\Rn)} \log \| F(u)\|_{L^1(\Rn)}\right| \leq C \|u\|_{q,V,w_0}^N\left|\log \| u\|_{q,V,w_0}\right|  .\]
Moreover, combining $(f_1)$ with elementary estimates for the function $t|\log t|$, and since  $q>N$, if $p\geq \frac{1}{N-1}$ we have 
\begin{multline*} \int_{\Rn} F(u) \log F(u) \, dx  \leq C\int_{\Rn} F(u) |\log F(u)| \, dx\\
 \leq C\left(\int_{\{|u|<1\}}|u|^N+\int_{\{|u>1|\}}F(u)\log|F(u)|\, dx \right)\\
\leq C\left(\|u\|_V^N+\int_{\{|u>1|\}}|u|^{p+1}e^{\alpha_N|u|^{\frac{N}{N-1}}}dx\right) \leq  C\left( \|u\|_V^N +\|u\|_V^{p+1}\right),
\end{multline*}
for small $\|u\|_V$, where we have also applied the classical Moser inequality on the whole $\Rn$. Similarly, when $p<\frac{1}{N-1}$ we obtain
\begin{multline*}\int_{\Rn} F(u) \log F(u) \, dx  \leq C\int_{\Rn} F(u) |\log F(u)| \, dx\\
\leq C\left(\int_{\{|u|<1\}}|u|^N+\int_{\{|u>1|\}}F(u)\log|F(u)|\, dx \right)\\
\leq C\left(\|u\|_V^N+\int_{\{|u>1|\}}|u|^{\frac{N}{N-1}}e^{\alpha_N|u|^{\frac{N}{N-1}}}dx\right) \leq  C \left(\|u\|_V^N + \|u\|_V^{\frac{N}{N-1}}\right).
\end{multline*}
 Combining the two previous estimates we end up with the following 
 \begin{multline*} \mathcal F (u) \leq C \| u\|_{q,V,w_0}^{N}\left( \|u\|_{q,V,w_0}^{N}+\|u\|_{q,V,w_0}^{N}|\log\|u\|_{q,V,w_0}|+\|u\|_{q,V,w_0}^{\frac{N}{N-1}}\right)\\
 \leq C\|u\|_{q,V,w_0}^{\frac{N^2}{N-1}} \ .
 \end{multline*}
 Hence, for $\rho$ small enough one has
 \[ \mathcal I_V (u) \geq \frac{1}N \| u\|_{q,V,w_0}^{N} - C\| u\|_{q,V,w_0}^{\frac{N^2}{N-1}} =\delta_0>0,\]
 with $\delta_0$ depending only on $\rho$, which proves $(i)$.

 \noindent In order to prove $(ii)$, let us consider a smooth function $e \in W^{1,N}_{q,V,w_0}(\Rn)$, supported in $B_{1/4}$. Since $F(e(x)) , F(e(y)) \neq 0$ only for $x, y \ B_{1/4}$, let us evaluate
 	\begin{multline*}
	\mathcal F(e) = \frac 12 \int_{\Rn} (I_N *F(e))(x) F(e(x)) \, dx\\
 	=  \frac{ 1}{2\gamma_N} \int_{\Rn}
 	    \left(   \int_{\Rn} \log\frac{ 1}{|x-y|} F(e(y) ) \, dy \right) F(e(x)) \,dx
 	  \\
 	\geq\frac{\log 2}{2\gamma_N} \left( \int_{ \{ |x| \leq \frac{1}4 \} } F(e)\, dx\right)^2,
 	\end{multline*}
 	from which we have 
 	\begin{multline*}
	\mathcal I_V(t e) = \frac{1}N t^N \|e\|_V^N - \mathcal F(te)\\
	 \leq    \frac{1}N t^N \|e\|_V^N - \frac{\log 2} {2\gamma_N} \left(\int_{\{ |x| \leq \frac{1}4 \} } F(te)\, dx\right)^2 \to -\infty,\end{multline*}
 	as $t\to+\infty$, since $F$ has exponential growth.
\end{proof}

\noindent By Ekeland's Variational Principle, there exists a Palais-Smale (PS in the sequel) sequence $\{ u_n\} \in W^{1,N}_V L^q_{w_0}(\Rn)$ such that
	\[
 		\mathcal I_V'(u_n) \to 0, \qquad \mathcal I_V (u_n) \to m_V,
 	\]
where $m_V$ is the mountain pass level,
	\[
	 0<m_V := \inf_{\gamma\in \Gamma} \max_{ t\in [0,1]} \mathcal I_V(\gamma(t)),
	 \]
	and
	\[
	\Gamma= \big\{ \gamma \in C^1\left([0,1], W^{1,N}_{q,V,w_0} (\Rn) \right)\, \big| \, \, \gamma (0)=0, \, \mathcal I_V(\gamma(1)) <0\} .
	\]
	
	\noindent Next, a few efforts are needed to extend to the higher dimensional case $N\geq 3$, the mountain pass level estimates carried out in \cite[Lemma 5.2]{castar}. 	
	\begin{lemma}\label{MPlevel-estimate}
	The mountain pass level $m_V$ satisifes
	$$m_V < \frac{1}N\ .$$
\end{lemma}

\begin{proof}
We are reduced to exhibit a function $v\in W^{1,N}_V L^q_{w_0}(\Rn)$ with unitary norm and such that 
	\[ \max_{t\geq 0} \mathcal I_V(t v) <\frac{1}N.\]
	For this purpose let us introduce the following Moser type functions for all $n\geq 1$, supported in $B_\rho$ for some $\rho>0$,
		\sys[ \overline w_n= ]{& C_n \log n , && 0\leq |x| \leq \frac{\rho}n,
									\\	
									&C_n \log\frac{\rho}{|x|},  && \frac{\rho}n \leq |x| \leq \rho,
									\\
									&0 , &&  \rho \leq |x|,}
									with \eqlab{ \label{cost} C_n
									= ( \omega_{N-1}\log n) ^{-\frac{1}N}.
									}
		We have 
		\[
			 \int_{\Rn} |\nabla \overline w_n|^N \, dx
		 = \omega_{N-1} C_n^N \int_{\frac{\rho}n}^\rho r^{-1}\, dr
		 =  \omega_{N-1} C_n^N \log n =1,
		 \] 							
as well as 
\begin{multline*}
		 \int_{\Rn}  V(x)  |\overline w_n|^N\, dx \\
		 \leq \sup_{B_\rho} V \left( C_n^N (\log n)^N \int_{B_{\rho/n} } \, dx
		 	+ \int_{B_\rho \setminus B_{\rho/n} } C_n^N\left( \log \frac{\rho}{|x|} \right)^N\, dx  \right)
		 	\\
		 	=C_n^N   \omega_{N-1} \sup_{B_\rho} V \left((\log n)^N \int_0^{\frac{\rho}n} r^{N-1} \, dr + \int_{\frac{\rho}n}^\rho \left( \log\frac{\rho}r \right)^N  r^{N-1} \, dr \right)
		 	\\
\leq\frac{\sup_{B_\rho} V }{\log n}    \left( \frac{(\log n)^N  }{N n^N} \rho^N + \int_{\frac{\rho}n}^\rho \left( \log\frac{\rho}r \right)^N r^{N-1}\,dr\right),
\end{multline*}
recalling from \eqref{cost} that $C_n^N \omega_{N-1} = 1/\log n$.
Now
\begin{multline*}\int_{\Rn} |\overline w_n|^q \log(1+|x|) \, dx =( C_n \log n)^q \int_{B_{\rho/n}} \log(1+|x|) \, dx \\
+ C_n^q \int_{B_\rho \setminus B_{\rho/n}} \log (1+|x|) \left| \log \frac{\rho}{|x|} \right|^q \, dx
\\
 	=C_n^q \omega_{N-1} \left( \log^q n \int_0^{\frac{\rho}n}\hspace{-.2cm}\log(1+r) r^{N-1} \, dr
 		 +  \int_{\frac{\rho}n}^\rho \hspace{-.2cm}\log (1+r)  \log^q \left(\frac{\rho}{r}\right) r^{N-1}\, dr\right)
 		 \\
 		 \leq  C_n^q \omega_{N-1}  \log^q n \left(  \frac{\rho}n \right)^{N+1}
 		+C_n^q \omega_{N-1}   \log(
 		e+\rho) \int_{\frac{\rho}n}^\rho  \left( \log \frac{\rho}{r} \right)^q r^{N-1}\, dr.
 		\end{multline*}
Thus,
\begin{multline*}	\|\overline w_n\|_{q,V,w_0}^N = \int_{\Rn}\left( |\nabla \overline w_n|^N + V |\overline w_n|^N \right)\, dx + 	\\
\left(\int_{\Rn} |\overline w_n|^q \log(1+|x|) \, dx\right)^{\frac{N}q}
		\\
		\leq 1 +  \frac{\sup_{B_\rho} V }{\log n}   \int_{\frac{\rho}n}^\rho \left( \log\frac{\rho}r \right)^N r^{N-1}\,dr + O\left( \frac{(\log n)^{N-1} }{n^N}\right)
		\\
		+  \frac{\omega_{N-1}^{\frac{N-q}q} }{\log n} ( \log(1+\rho) )^{\frac{N}q}\left( \int_{\frac{\rho}n}^\rho  \left( \log \frac{\rho}{r} \right)^q r^{N-1}\, dr\right)^{\frac{N}q}\\
		  + O \left( \frac{1}{\log n} \left( \frac{ \log n}{n^{\frac{N+1}q} }\right)^N\right)\ .
		\end{multline*}
Let us estimate explicitly integrals in the above inequality, as for $k\in \N$ we have 
	\[
	\int \left(\log \frac{\rho}r \right)^k r^{N-1} \, dr
	= \frac{ r^N}N \sum_{j=0}^k \left(\log \frac{\rho}r \right)^{k-j} \frac{k(k-1) \dots (k-j+1)}{N^j},
	\]
	and from the estimate
			\begin{multline*}
			 \int \left(\log \frac{\rho}r \right)^q r^{N-1} \, dr
			 \leq 	\int \left(\log \frac{\rho}r \right)^{[q]} r^{N-1} \, dr
			+	\int \left(\log \frac{\rho}r \right)^{[q+1]} r^{N-1} \, dr,
			\end{multline*}	
			we get 	
		\[
			1\leq \|\overline w_n\|_{q,V,w_0}^N \leq 1+ \delta_n,\qquad \mbox{ with } \delta_n \to 0, \mbox{ as } n\to \infty
			\]
			and
			\begin{multline}\label{deltan}
				\delta_n= \frac{1}{\log n} \frac{\rho^N}N \left[  \sup_{B_\rho} V  \frac{N! }{N^N}
				+ \omega_{N-1}^{\frac{N-q}q}  ( \log(e+\rho) )^{\frac{N}q} \left( \frac{[q]! }{N^{[q]} }+\frac{[q+1]!}{N^{[q+1]}}\right) \right]\\
				+o\left( \frac{1}{\log n}\right) \ .
\end{multline}
			Then,
			\[ w_n =\frac{\overline w_n}{\sqrt[N]{1+\delta_n} }, \qquad \|w_n\| _{q,V,w}^N \leq 1.\]
			\noindent \textit{Claim:}
			\eqlab{ \label{cl1} \exists \, n \in \N \mbox{ such that }  \max_{t\geq 0}  \mathcal I_V(t w_n) < \frac{1}{N}.}
			By contradiction, suppose that for all $n$
			\[ \mathcal I_V(t_n w_n) := \max_{t\geq 0} \mathcal I_V(t w_n) \geq \frac{1}{N},\]
			together with 
			\[
				\frac{d}{dt} \mathcal I_V(t w_n)|_{t=t_n} =0 \ .
			\]
As a consequence we obtain 
	\eqlab{  \label{uf1}
	& \frac{t_n^N}{N} \geq \frac{1}{N} +\frac 1{2\gamma_N} \int_{\R^{2N}} \log\frac{1}{|x-y|} F(t_n w_n(y)) F(t_n w_n(x)) \, dx \, dy,
	}
	and
	\eqlab{ \label{uf2}
	t_n^N\geq \frac 1{\gamma_N}\int_{\R^{2N}} \log\frac{1}{|x-y|} F(t_n w_n(y)) f(t_n w_n(x)) t_n w_n(x) \, dx \, dy.
	}
	Assume $\rho \leq 1/2$, thus if $x,y \in B_\rho$, then $\log (1/|x-y|)\geq 0$. Observe from \eqref{uf1} that, since $w_n$ is supported in $B_\rho$,
	\[ t_n \geq 1 .\]
	Next we prove the following 
		\eqlab{ \label{inf1}
			\liminf_{n \to +\infty} t_n \leq 1\ .
			}
			Indeed, if not there exists some $\delta_n>0$ such that for $n$ large enough
\eqlab{ \label{oo}  t_n^N \geq 1+\delta_n \ .}			
			Notice that
			\bgs{
			I:= &\;  \int_{\R^{2N}} \log\frac{1}{|x-y|} F(t_n w_n(y)) f(t_n w_n(x)) t_n w_n(x) \, dx \, dy \\
			 = &\;
			\int_{B_{\frac{\rho}n} \times B_{\frac{\rho}n} } \log\frac{1}{|x-y|} F(t_n w_n(y)) f(t_n w_n(x)) t_n w_n(x) \, dx \, dy
			\\
			&\; +
			 \int_{\R^{2N} \setminus \left(B_{\frac{\rho}n} \times B_{\frac{\rho}n} \right)} \log\frac{1}{|x-y|} F(t_n w_n(y)) f(t_n w_n(x)) t_n w_n(x) \, dx \, dy
			 \\
			 \geq &\;
			 \int_{B_{\frac{\rho}n} \times B_{\frac{\rho}n} } \log\frac{1}{|x-y|} F(t_n w_n(y)) f(t_n w_n(x)) t_n w_n(x) \, dx \, dy,
			 }
			 again since   $\log (1/|x-y|)\geq 0$ for $x,y \in B_\rho$, where $w_n$ is supported. By $(f_4)$, for any $\eps\in (0,\beta/2)$ small enough, there exists $s_\eps>0$ such that for all $s>s_\eps$,
			 \[
			 s^{\frac{2N-1}{N-1}} f(s) F(s) > \frac{\beta}2 e^{ 2N \omega_{N-1}^{\frac{1}{N-1} } s^{\frac{N}{N-1}}}\] 
			 and in turn 
			 \[s f(s) F(s) > \frac{\beta}2 s^{-\frac{N}{N-1}} e^{ 2N \omega_{N-1}^{\frac{1}{N-1} } s^{\frac{N}{N-1}}}.
			 \]
			 Therefore, by using the explicit value of $w_n$ and the fact that $|x-y|\leq 2\rho/n$ for $x, y\in B_{\frac{\rho}n} \times B_{\frac{\rho}n} $, we have 
			\begin{multline*}
			 I
			 \geq  \frac{\beta}2\log \left(\frac{n}{2\rho} \right)
			 \frac{\omega_{N-1}^2}{N^2}
			 \left(\frac{\rho}n \right)^{2N}
			 e^{2N \omega_{N-1}^{\frac{1}{N-1}} \left(\frac{t_n C_n \log n}{\sqrt[N]{1+\delta_n}}\right)^{\frac{N}{N-1}} }
			  \left(\frac{t_n C_n \log n}{\sqrt[N]{1+\delta_n}}\right)^{\frac{-N}{N-1}}
			 \\
			= \frac{\beta}2  \frac{\omega_{N-1}^2}{N^2}\log
			\left(\frac{n}{2\rho} \right) \rho^{2N} e^{-2N \log n }
			e^{2N  \log n\, \left(\frac{t_n}{\sqrt[N]{1+\delta_n}}\right)^{\frac{N}{N-1} } } \cdot
			\\ \cdot \left(\frac{t_n}{\sqrt[N]{1+\delta_n}}\right)^{\frac{-N}{N-1} }\omega_{N-1}^{\frac{1}{N-1}} \frac{1}{\log n}
			\\
			\geq 
			\frac{\beta}{2N^2}  \omega_{N-1}^{\frac{2N-1}{N-1}} \rho^{2N} \frac{e^{2N  \log n \left[ \left(\frac{t_n}{\sqrt[N]{1+\delta_n}}\right)^{\frac{N}{N-1} } -1\right]} }{ \left(\frac{t_n}{\sqrt[N]{1+\delta_n}}\right)^{\frac{N}{N-1}}},
			 \end{multline*}
			 recalling the value of $C_n$ given in \eqref{cost}. It follows from \eqref{uf2} that
			\eqlab{ \label{off2}
						 t_n^N \geq \frac{\beta}{2N^2\gamma_N}   \omega_{N-1}^{\frac{2N-1}{N-1}} \rho^{2N} \frac{e^{2N  \log n \left[ \left(\frac{t_n}{\sqrt[N]{1+\delta_n}}\right)^{\frac{N}{N-1} } -1\right]} }{ \left(\frac{t_n}{\sqrt[N]{1+\delta_n}}\right)^{\frac{N}{N-1}}}.
			 }
			 In  both cases when $t_n \to \infty$, as $n\to \infty$ or when $t_n$ stays bounded, \eqref{oo} yields a contradiction. Thus \eqref{inf1} holds and hence
			 \[ \lim_{n\to \infty} t_n=1 \ .\]

		 \noindent Now, from one side we have
		 \[
			 	 e^{ \log n \left[ \left(\frac{t_n}{\sqrt[N]{1+\delta_n}}\right)^{\frac{N}{N-1} } -1\right]}  \leq C\] 
				 and thus \[ \left(\frac{t_n}{\sqrt[N]{1+\delta_n}}\right)^{\frac{N}{N-1} }  \leq 1+ O\left(\frac{1}{\log n}\right)
			 \ .\] 
			 
			\noindent  On the other side, from \eqref{off2} we obtain
	\[
		1+ o(1) \geq t_n^{\frac{N^2}{N-1}} \geq \frac{\beta}{2N^2\gamma_N}   \omega_{N-1}^{\frac{2N-1}{N-1}} \rho^{2N} e^{2N  \log n \left[ \left(\frac{t_n}{\sqrt[N]{1+\delta_n}}\right)^{\frac{N}{N-1} } -1\right]} (1+\delta_n)^{\frac{1}{N-1}}
	.\]
	As a consequence we finally get 
	\bgs{
		1+ o(1) \geq &\;\frac{\beta}{2N^2\gamma_N}   \omega_{N-1}^{\frac{2N-1}{N-1}} \rho^{2N} e^{2N  \log n \left[ \left(\frac{t_n}{\sqrt[N]{1+\delta_n}}\right)^{\frac{N}{N-1} } -1\right]} (1+\delta_n)^{\frac{1}{N-1}}
		\\
		=&\;  \frac{\beta}{2N^2\gamma_N}   \omega_{N-1}^{\frac{2N-1}{N-1}} \rho^{2N} e^{2N  \log n \left[ -\frac{\delta_n}{N-1} + o(\delta_n) \right]}.
		}
	By substituting \eqref{deltan} in the previous inequality and letting $n\to \infty$, we end up with 
		\bgs{
			1\geq &\;\frac{\beta}{2N^2\gamma_N}  \omega_{N-1}^{\frac{2N-1}{N-1}} \rho^{2N} e^{ \frac{-2N}{N-1} \frac{\rho^N}N \left[  \sup_{B_\rho} V  \frac{N! }{N^N}
				+ \omega_{N-1}^{\frac{N-q}q}  ( \log(1+\rho) )^{\frac{N}q} \left( \frac{[q]! }{N^{[q]} }+\frac{[q+1]!}{N^{[q+1]}}\right) \right]
				} .
		}
		For a fixed $\rho \leq 1/2$, set 
		\[\upnu :=\sup_{\rho\leq \frac12}  \frac{2N^2\gamma_N}{\rho^{2N}}  \omega_{N-1}^{\frac{-2N+1}{N-1}} e^{ \frac{2N}{N-1} \frac{\rho^N}N \left[  \sup_{B_\rho} V  \frac{N! }{N^N}
				+ \omega_{N-1}^{\frac{N-q}q}  ( \log(1+\rho) )^{\frac{N}q} \left( \frac{[q]! }{N^{[q]} }+\frac{[q+1]!}{N^{[q+1]}}\right) \right]
				}  
				\]
		 to get a contradiction from $(f_4)$, since 
		$ \beta >\upnu $. 
	\end{proof}

\subsection{On the Ekeland Palais-Smale sequence}

\noindent In this section we study the behavior of the PS sequence provided by Ekeland's Variational Principle. In particular, it is a non trivial fact, in this context, that the weak limit turns out to be a nontrivial solution of the equation. Boundedness of PS sequences buys the line of  \cite[Lemma 6.1]{castar}, to which we refer for the proof of the next 
\begin{lemma}\label{lem-PSbdd}
	Assume $(V)$ and $(f_1)$--$(f_4)$. Let
	$\{u_n\}\subset W^{1,N}_V L^q_{w_0}(\Rn)$ be an arbitrary PS sequence for $\mathcal I_V$ at level $c$, namely
	$$
	\mathcal I_V(u_n)\to c \quad \text{ and } \quad \mathcal I_V'(u_n)\to 0 \quad \hbox{in }
	 \left(W^{1,N}_V L^q_{w_0}(\mathbb R^N)\right)', \quad \hbox{ as } n\to +\infty,
	$$
	the dual space of $W^{1,N}_V L^q_{w_0}(\Rn)$.
	Then, the following hold:
\begin{itemize}
\item[$(i)$]	$\|u_n\|_V\leq C$\ ;
\item[$(ii)$]
	\begin{equation*}
	\left|\int_{\R^N}\left[\log \frac{1}{|x|}\ast F(u_n)\right]F(u_n) dx\right|\leq C\ ;
	\end{equation*}	
\item[$(iii)$]
	\begin{equation*}
	\left|\int_{\R^N}\left[\log \frac{1}{|x|}\ast F(u_n)\right]u_nf(u_n) dx\right|\leq C\ .
	\end{equation*}
\end{itemize}
	\end{lemma}

\begin{remark}
	Note that, as a consequence of Lemma \ref{lem-PSbdd}, we may assume the PS sequence at level $c$ to be positive. Indeed, since ${u_n}$ is bounded, we can test $\mathcal I'_V(u_n)$ against $u_n^{-}=\max{(-u_n,0)}$ to get 
	$$
	\left|\int_{\{u_n<0\}}|\nabla u_n|^N+V|u_n|^{N} dx\right|\leq \tau_n C
	$$
	Hence, the positive sequence $\{u_n^{+}\}$ is still a PS sequence at the same level $c$, since $F(s)=0$ for $s\leq 0$.
	\end{remark}
	\noindent From now on we will consider only positive PS sequences. Because of the exponential nonlinearity and the presence of a sign-changing logarithmic kernel, we cannot exploit standard arguments to obtain the existence of a solution as byproduct of boundedness of a PS sequence. Here it is fundamental to take advantage of the key estimate for the mountain pass level of Lemma \ref{MPlevel-estimate}. The next lemma is an extension of \cite[Lemma 6.2]{castar}. However, it is not a virtual transcription, so that, for convenience of the reader, we recall the main steps of the proof.

\begin{lemma}\label{lem-PSimproveTM}
	Assume $(V)$ and $(f_1)$--$(f_4)$. Let
	$\{u_n\}\subset W^{1,N}_VL^q_{w_0}(\Rn)$ be a (positive) PS sequence for $\mathcal I_V$ at level $0<c<1/N$. Then,  for any $1\leq \alpha <1/(Nc)$ the following uniform bound holds
	\begin{equation*}
	\sup_{n\in \mathbb N}\int_{\R^N}F^\alpha (u_n)\, dx<\infty\ .
	\end{equation*}
\end{lemma}
\begin{proof}
From Lemma \ref{lem-PSbdd}, there exists $u\in W_V^{1,N}(\Rn)$ such that:
	\bgs{ & u_n\rightharpoonup u &&  \mbox{ in }  W_V^{1,N}(\Rn); \\
	&  u_n\to u && \mbox { in } L^s_{\loc}(\Rn) \; \; \mbox{ for any  } 1\leq s< \infty;
	\\
	& u_n\to u && \mbox{ a.e. in }  \Rn,
	}
	with
	\begin{equation} \label{fru}
	\lim_{n\to +\infty} \|u_n\|_V^N=A^N\geq \|u\|_V^N.
	\end{equation}
	Let $G\colon\R^+ \to \R^+$,
	$$
	G(t):=\int_0^t \sqrt[N]{\frac N2\frac{F(s)f'(s)}{f^2(s)}-\frac{N-2}{2}}ds,
	$$
	and notice that $G\in  C^1(\R^+)$ thanks to $(f_2)$.  By H\"{o}lder's inequality we have
	\begin{multline}\label{estG}
	G^N(t)\leq \left(\int_0^t ds\right)^{N-1} \int_0^t \left(\frac N2\frac{F(s)f'(s)}{f^2(s)}-\frac{N-2}{2}\right)ds\\
	=t^N-\frac N2t^{N-1}\frac{F(t)}{f(t)} \ .
	\end{multline}
Set
	$$v_n:=G(u_n)>0\ .$$
Since $u_n$ is bounded in $W_V^{1,N}(\Rn)$ and thanks to $(f_2)$, 
	\begin{equation*}
	\int_{\Rn}|\nabla v_n|^N dx=\int_{\Rn}|\nabla u_n|^N\left(\frac N2\frac{F(u_n)f'(u_n)}{f^2(u_n)}-\frac{N-2}{2}\right) dx\leq C
	\end{equation*}
	and
	\[
		\int_{\Rn}Vv_n^N dx=\int_{\Rn}V G^N(u_n) dx\leq C \int_{\Rn} V u_n^N \, dx \leq C.
		\]
We claim that for $n$ large enough
	$$\|\nabla v_n\|_V^N\leq 1.$$
		
	\noindent Combining the facts $\mathcal I_V(u_n) \to c$,
	\eqref{fru1} and \eqref{fru}, we have
	\begin{equation*}
	\lim_{n\to +\infty}\frac 1{\gamma_N}\int_{\Rn}\left[\log\left(\frac{1}{|x|}\right)\ast
	F(u_n)\right]F(u_n) dx=2 \left( \frac{A^N}N- c\right).
	\end{equation*}
	Moreover, since $\mathcal I_V'(u_n) \to 0$ in $\left(W^{1,N}_VL^q_{w_0}(\Rn)\right)'$, we have 
	\[\mathcal I_V'(u_n)\left[ \frac{F(u_n)}{f(u_n)}\right] \to 0,\]
	and hence
	\eqlab{ \label{fru2}
	&\int_{\Rn}|\nabla u_n|^N\left(1-\frac{F(u_n)f'(u_n)}{f^2(u_n)}\right) dx + \int_{\Rn}Vu_n^{N-1}\frac{F(u_n)}{f(u_n)} dx
	 \\
	&	-\frac{1}{\gamma_N} \int_{\Rn}\left[\log\left(\frac{1}{|x|}\right)\ast
	F(u_n)\right]F(u_n) dx = {\rm{o}}(1).
	}
Again by \eqref{fru} we get 
	\bgs{
	&\int_{\Rn}|\nabla u_n|^N\left(1-\frac{F(u_n)f'(u_n)}{f^2(u_n)}\right) dx+\int_{\Rn}V u_n^{N-1}\frac{F(u_n)}{f(u_n)} dx+2c\\
	- &\;\frac 2N\int_{\Rn}|\nabla u_n|^N dx -\frac 2N\int_{\Rn}V u_n^{N} dx = {\rm{o}}(1).
	}
	Therefore, thanks to \eqref{estG}
	\eqlab{ \label{normvn1}
	\|v_n\|_V^N= &\;\int_{\Rn}|\nabla G(u_n)|^N dx +\int_{\Rn}VG^N(u_n) dx
	\\
	=&\; \int_{\Rn} |\nabla u_n|^N \left(\frac N2\frac{F(u_n) f'(u_n)}{f^2(u_n)}-\frac{N-2}2\right) \, dx+ \int_{\Rn}VG^N(u_n) dx
	\\
	= &\;Nc+\int_{\Rn} V\left(\frac N2u_n^{N-1}\frac{F(u_n)}{f(u_n)}-u_n^N+G^N(u_n)\right) dx+{\rm{o}}(1)
	\\
		\leq &\; Nc+{\rm{o}}(1) <1}
	for $n$ large enough.

	\noindent At this point, we are able to improve the exponential integrability of $u_n$. Thanks to $(f_3)$, for any $\epsilon>0$ there exists $t_\epsilon>0$ large enough such that
	$$
	1-\epsilon<\sqrt[N]{\frac N2\frac{F(t)f'(t)}{f^2(t)}-\frac{N-2}{2}}\leq 1+\epsilon, \quad \text{ for all }\ t\geq t_\epsilon\ .
	$$
	Next by  $(f_2)$ we also have either $u_n(x)\leq t_\epsilon $ or $u_n(x)\geq t_\epsilon$ which implies
	\begin{multline}\label{vn-un}
	v_n\geq \int_{0}^{t_\epsilon}\sqrt[N]{\delta \frac N2} dt + \int_{t_\epsilon}^{u_n}(1-\epsilon)dt\\
	\geq \sqrt[N]{\delta \frac N2} t_\epsilon +(1-\epsilon)(u_n-t_\epsilon)\geq (1-\epsilon)(u_n-t_\epsilon)
	\end{multline}
	and in turn 
	$$
	u_n\leq t_\epsilon + \frac{v_n}{1-\epsilon}, \quad \hbox{ for any } x\in \mathbb R^N \ .
	$$
	Hence, by \eqref{f11} and since $F$ is an increasing function, and recalling that $q>N$,
	\begin{multline}\label{Falpha1}
	\int_{\Rn} F^\alpha(u_n) dx = \int_{0\leq u_n\leq t_\epsilon} F^\alpha(u_n) dx + \int_{ u_n\geq t_\epsilon} F^\alpha(u_n) dx\\
	\leq C_\epsilon\int_{u_n\leq t_\epsilon}u_n^{N\alpha} dx \\
	+\int_{u_n\geq t_\epsilon}\left[ F\left(t_\epsilon + \frac{v_n}{1-\epsilon}\right)\right]^\alpha dx\\
	\leq C_\epsilon \int_{u_n\leq t_\epsilon}u_n^{N}dx +C\int_{u_n\geq t_\epsilon}\hspace{-.2cm}  \left(t_\epsilon + \frac{v_n}{1-\epsilon}\right)^{\alpha(p-\frac 1{N-1})}\hspace{-.6cm} \phi_N \left( \alpha \alpha_N (t_\varepsilon + \frac{v_n}{1-\epsilon})^{\frac{N}{N-1}}\right) dx\\
	\leq  C_\epsilon \|u_n\|_N^N+C_\epsilon\int_{u_n\geq t_\epsilon}\phi_N\left( \alpha \alpha_N(1+\epsilon)(t_\epsilon + \frac{v_n}{1-\epsilon})^{\frac N{N-1}}\right)dx\\
\leq	  C_\epsilon \|u_n\|_N^N+C_\epsilon\int_{\Rn}
\phi_N\left(\alpha \alpha_N (1+\epsilon)^\frac{N}{N-1}\frac{v_n^{\frac N{N-1}}}{(1-\epsilon)^{\frac{N}{N-1}}}\right)dx,
	\end{multline}
	where $C_\epsilon>0$ may change from line to line. (The last inequality can be verified just observing that for large values of $u_n$, also $v_n$ is large so that  $(t_\epsilon + \frac{v_n}{1-\epsilon})^{\frac N{N-1}}\sim (\frac{v_n}{1-\epsilon})^{\frac N{N-1}} $).
	
	\noindent 	Set
	$$
	\eta:= \frac 1{Nc}-\alpha >0$$
	and let us  fix $0<\epsilon_\alpha<1$, depending on $\alpha<\frac 1{Nc}$  such that
	$$
	\frac{(1+\epsilon_\alpha)^{\frac N{N-1}}}{(1-\epsilon_\alpha)^{\frac N{N-1}}}\left(1-\eta^2(Nc)^2\right)<1 .
	$$
	With these choices we obtain
	\begin{multline*}
	\int_{\Rn}F^\alpha(u_n)
	 dx\leq C_\alpha \|u_n\|_N^N+\\
	 C_\alpha\int_{\Rn}
	\phi_N\left(\alpha \alpha_N \frac{(1+\epsilon_\alpha)^\frac{N}{N-1}}{(1-\epsilon_\alpha)^\frac{N}{N-1}}\|v_n\|_V^{\frac{N}{N-1}}\frac{|v_n|^{\frac N{N-1}}}{\|v_n\|_V^{\frac{N}{N-1}}}\right)dx\ .
	\end{multline*}
	By \eqref{normvn1}, $\|v_n\|_V^N\leq Nc+\rm{o}(1)$ as $n$ is large enough, so that
	$$
	\|v_n\|_V^N\leq Nc+(Nc)^2\eta,  \quad \hbox{ as } n\to +\infty\ .
	$$
	Thus,
	\begin{multline*}
	\alpha\frac{(1+\epsilon_\alpha)^{\frac N{N-1}}}{(1-\epsilon_\alpha)^{\frac N{N-1}}}\|v_n\|_V^N\leq  \left(\frac 1{Nc}-\eta\right)	\frac{(1+\epsilon_\alpha)^{\frac N{N-1}}}{(1-\epsilon_\alpha)^{\frac N{N-1}}}Nc(1+Nc\eta)\\
	=	\frac{(1+\epsilon_\alpha)^{\frac N{N-1}}}{(1-\epsilon_\alpha)^{\frac N{N-1}}}\left(1-(Nc)^2\eta^2\right)<1,
	\end{multline*}
	and finally we obtain
	$$
	\int_{\Rn}F^\alpha(u_n) dx\leq C_\alpha \|u_n\|_N^N+C_\alpha\int_{\Rn} \phi_N\left(\alpha_N \frac{|v_n|^{\frac N{N-1}}}{\|v_n\|_V^{\frac N{N-1}}}\right) dx \leq C_{\alpha}\ .
	$$

\end{proof}
\begin{proposition}\label{prop-PS}
	Assume that conditions $(V)$ and $(f_1)$--$(f_4)$ are satisfied. Let
	$\{u_n\}\subset W^{1,N}_{q,w}$ be a PS sequence for $\mathcal I_V$ at level $c<1/N$, weakly converging to $u$ in $W_V^{1,N}(\Rn)$. If $u\neq 0$, then $u\in W^{1,N}_VL^q_{w_0(\Rn)}$ and
	$u_n\rightharpoonup u$ weakly in  $W^{1,N}_VL^q_{w_0(\Rn)}$. Furthermore, as $n\to\infty$
	\begin{eqnarray} \label{convFF}
	\Big(\log |x| \ast F(u_n)\Big)f(u_n)\longrightarrow\Big(\log
	|x|\ast F(u)\Big)f(u) \quad \hbox{ in } L_{loc}^1(\Rn)
	\end{eqnarray}
	and $u$ is a weak solution to \eqref{main}.
\end{proposition}
\noindent For the proof we refer to \cite[Proposition 6.3]{castar}.

	\subsection{Proof  of Theorem \ref{thm1}}\label{final_sec}
	
 \noindent  The functional $\mathcal I_V$ satisfies the Mountain Pass geometry thanks to  Lemma \ref{lemMP}. This yields a (PS) sequence  $\{u_n\}\subset W^{1,N}_VL^q_{w_0}(\Rn)$ at level ${m_V}$. Then, by Lemma \ref{lem-PSimproveTM} we have that $\{u_n\}$ is bounded in $u\in W_V^{1,N}(\Rn)$ and it weakly converges to some  $u\in W_V^{1,N}(\Rn)$. It remains to prove that $u\neq 0$. 

\noindent Either $\{u_n\}$ is vanishing, that is for any $r>0$
\[
\lim_{n\to +\infty}\sup_{y\in \Rn}\int_{B_r(y)}|u_n|^N dx=0
\]
or, there exist $r, \delta >0$ and a sequence $\{y_n\}\subset \mathbb Z^N$ such that
\[
\lim_{n\to \infty}\int_{B_r(y_n)}|u_n|^N dx\geq \delta.
\]
If $\{u_n\}$ is vanishing, by Lions' concentration-compactness principle we have
\begin{equation}\label{Lions}
u_n\to 0 \quad \mbox{in} \quad  L^s(\Rn) \quad \forall \, s>N,
\end{equation}
as $n\to\infty$. In this case it is standard to show that
\[
\|F(u_n)\|_{\gamma}, \|u_nf( u_n)\|_{\gamma}\to 0
\]
for some values of $\gamma>1$ and close to 1, thanks to the improved exponential integrability given by Lemma \ref{lem-PSimproveTM} and the growth assumption $F(t)<\frac N2tf(t)$, see \eqref{F/f}). Hence, by applying the HLS inequality (Proposition \ref{HLS}) we obtain as $n\to\infty$, similarly to Proposition \ref{prop-PS}:
\begin{eqnarray}
\int_{\R^{2N}}\log\left(1+\frac{1}{|x-y|}\right) F(u_n(x))F(u_n(y))dxdy &\to 0& \label{fineq1}\\
\int_{\R^{2N}} \log\left(1+\frac{1}{|x-y|}\right) F(u_n(x))u_n(y)f(u_n(y)) dx dy &\to 0& \label{fineq2}
\end{eqnarray}
Combining \eqref{fineq1}-\eqref{fineq2} and the facts $\mathcal I_V(u_n) \to c$ and $\mathcal I'_V(u_n)[v]\to 0$ on $C^\infty_c(\Rn)$ test functions, we obtain
\begin{multline*}
\frac 1{\gamma_N}\int_{\R^{2N}} \log\left(1+|x-y|\right) F(u_n(x))\left[F(u_n(y))-\frac 2Nu_n(y)f(u_n(y))\right] dx dy\\
= 2m_V+{\text{o}}(1)
\end{multline*}
so that $m_v\leq 0$ thanks to \eqref{F/f}, which is not possible. Therefore, the vanishing case does not occur.

\medskip

\noindent Now set $v_n:=u_n(\cdot-y_n)$, then
\begin{equation}\label{nonvan}
\int_{B_r(0)}|v_n|^2 dx\geq \delta\ .
\end{equation}
Using the periodicity assumption, $\mathcal I_V$ and $\mathcal I'_V$ are both invariant by the $\mathbb Z^N$-action, therefore $\{v_n\}$ is
still a PS sequence at level $m_V$. Then $v_n\rightharpoonup v$ in $W_V^{1,N}(\Rn)$ with $v\neq 0$ by using \eqref{nonvan}, since $v_n\to v $ in $L^N_{loc}(\Rn)$. We conclude by Proposition \ref{prop-PS} that $v\in W^{1,N}_V L^q_{w_0}(\Rn)$ is a nontrivial critical point of $\mathcal I_V$ and $\mathcal I_V(v)=m_V$, which completes the proof of Theorem \ref{thm1}.

\end{document}